\documentclass[11pt, letterpaper]{amsart}
\usepackage[left=1in,right=1in,bottom=1in,top=1in]{geometry}
\usepackage{amsmath,amsthm,amssymb}
\usepackage{mathrsfs}
\usepackage[all,cmtip]{xy}
\usepackage{cite}
\usepackage{hyperref}
\usepackage{bm}

\newcommand{\on}[1]{\operatorname{#1}}
\newcommand{\mathfont}{\mathbf}

\newcommand{\ZZ}{\mathfont Z}

\newcommand{\QQ}{\mathfont Q}

\newcommand{\g}{\mathfrak{g}}
\newcommand{\h}{\mathfrak{h}}

\newcommand{\m} {\mathfrak m}

\DeclareFontFamily{OT1}{rsfs}{}
\DeclareFontShape{OT1}{rsfs}{n}{it}{<-> rsfs10}{}
\DeclareMathAlphabet{\mathscr}{OT1}{rsfs}{n}{it}

\newcommand{\Hcal}{\mathcal{H}}
\newcommand{\Ccal}{\mathcal{C}}

\newcommand{\Id}{\on{Id}}
\newcommand{\Hom}{\on{Hom}}

\newcommand \tensor[1] {\otimes_{#1}}
\renewcommand{\Im}{\on{Im}}

\newcommand{\Aut}{\on{Aut}}

\newcommand{\Gal}{\on{Gal}}

\newcommand{\End}{\on{End}}

\newcommand{\onto}{\twoheadrightarrow}

\newcommand{\Rep}{\on{Rep}}

\renewcommand{\O}{\mathcal{O}}

\newcommand{\rhobar}{\overline{\rho}}

\newcommand{\Spec}{\on{Spec}}

\newcommand{\Lie}{\on{Lie}}

\newcommand{\GL}{\on{GL}}
\renewcommand{\sl}{\mathfrak{sl}}

\newcommand{\Ad}{\on{Ad}}
\newcommand{\ad}{\on{ad}}

\newcommand{\Ga}{\mathfont{G}_a}
\newcommand{\Gm}{\mathfont{G}_m}

\theoremstyle{plain}
\newtheorem{lem}{Lemma}
\newtheorem{thm}[lem]{Theorem}
\newtheorem{prop}[lem]{Proposition}

\theoremstyle{definition}
\newtheorem{defn}[lem]{Definition}

\newtheorem{remark}[lem]{Remark}

\newcommand{\WD}{WD_G}
\newcommand{\WDL}{WD_{G, L}}
\newcommand{\lw}{l.w.}

\title{$G$-Valued Galois Deformation Rings when $\ell \neq p$}
\author{Jeremy Booher and Stefan Patrikis}
\date{\today}
\thanks{We are grateful to Rebecca Bellovin and Toby Gee for encouraging us to write up our work. We also thank Brian Conrad and Bhargav Bhatt for commutative algebra consultations, and the referee for helpful comments. S.P. was partially supported by NSF Grant DMS-1700759.}

\begin{document} 

 \begin{abstract}
For a smooth group scheme $G$ over an extension of $\ZZ_p$ such that the generic fiber of $G$ is reductive, we study the generic fiber of the Galois deformation ring for a $G$-valued mod $p$ representation of the absolute Galois group of a finite extension of $\QQ_\ell$ with $\ell \neq p$.  In particular, we show it admits a regular dense open locus, and that it is equidimensional of dimension $\dim G$.
 \end{abstract}
 
 \maketitle
 
\section{Introduction}

Many of the deepest arithmetic properties of modular forms (or more generally algebraic automorphic representations) are encoded in the congruences between eigenforms of different level or weight. These congruences can be interpreted as congruences between the corresponding modular Galois representations, and this perspective, in combination with the development of modularity lifting and potential automorphy theorems, has dramatically advanced our understanding of such congruences. The most successful approach to producing congruences is based on a method of Khare-Wintenberger \cite{kw} that in appropriate settings produces lifts of prescribed inertial type for a potentially modular mod $p$ representation $\rhobar$; when $\rhobar$ is in fact modular, the method combines with modularity lifting theorems to produce congruences between modular forms.

The present paper is a contribution to the local aspect of this story for Galois representations valued in general reductive groups. Namely, the method of \cite{kw} depends in part on having an adequate understanding of the structure of the generic fibers of local Galois deformation rings, and we focus our attention here. Let $\ell$ and $p$ be distinct primes, $K$ be a finite extension of $\QQ_{\ell}$, and $E$ be a finite extension of $\QQ_p$, with ring of integers $\O_E$ and residue field $k$. Consider a smooth group scheme $G$ over $\O_E$ with reductive generic fiber,\footnote{In particular, $G$ and its generic fiber need not be connected. As the definition of a reductive group scheme typically includes a connectedness hypothesis, we avoid that language.} let $\Gamma_K= \mathrm{Gal}(\overline{K}/K)$, and consider a continuous representation
\[
\rhobar \colon \Gamma_K \to G(k).
\]
The functor of lifts of $\rhobar$ to artin local $\O_E$-algebras with residue field $k$ is pro-represented by a complete local noetherian $\O_E$-algebra $R^\square_{\rhobar}$. This ring may be quite singular; the object of this paper is achieve some control over its \textit{generic fiber} $R^\square_{\rhobar}[\frac{1}{p}]$. The main result is the following. First we recall that to each homomorphism $\tau \colon I_K \to G(E)$ that factors through a finite quotient, we can associate a quotient $(R^\square_{\rhobar}[\frac{1}{p}])^\tau$ of $R^\square_{\rhobar}[\frac{1}{p}]$ (``with inertial type $\tau$'') which is a union of irreducible components of $R^\square_{\rhobar}[\frac{1}{p}]$.
\begin{thm}\label{main}
For any $\rhobar \colon \Gamma_K \to G(k)$ and $\tau$ as above, $(R^\square_{\rhobar}[\frac{1}{p}])^\tau$ admits a regular, dense open subscheme, and it is equidimensional of dimension $\dim(G)$.
\end{thm}

We will now give some of the history behind this result. Theorem \ref{main} generalizes a result of Gee \cite{gee11} in which the group $G$ is taken to be $\GL_n$. His result in turn is the adaptation to the $\ell \neq p$ case of a corresponding result of Kisin \cite{kisin08} that studies the generic fibers of potentially semi-stable deformation rings (in the case $\ell=p$); Kisin introduced many new ideas on which all subsequent work in this subject has been based. Later, Gee's results were reproved and slightly strengthened in \cite{choi09} and \cite{blggt14}, and in  \cite{bellovin16} Bellovin proved a version of Kisin's $\ell=p$ result for general $G$.

\begin{remark}
The results of \cite[Theorem 3.3.4]{kisin08}, and later papers like \cite[Theorem 2.1.6]{gee11} that adapt its arguments, state there is a locus in the generic fiber that is formally smooth over $\QQ_p$, instead of stating it is regular.  As discussed in Remark~\ref{rmk:formalsmoothness}, regularity does not imply formal smoothness in this setting, and so these results are incorrect as stated.  However, this distinction is irrelevant for the applications to computing the dimension of the generic fiber, as all that is needed is regularity.
\end{remark}

While we were preparing this paper, Bellovin and Gee posted a preprint that treats both the cases $\ell=p$ and $\ell \neq p$ for general groups $G$ \cite{bg17} . In addition to treating the more difficult case $\ell=p$, their work gives a more refined version of our Theorem \ref{main} (finding a dense set of ``very smooth" points), motivated by the refinement of \cite{gee11} established and applied in \cite{blggt14}. Moreover, they extensively pursue global applications to lifting results of Khare-Wintenberger type, including applications to generalizations of the Serre weight conjecture: see the theorems in the introduction of \cite{bg17}.

Because the work in \cite{bg17} is so comprehensive, and because of the timing of our two papers, we have not attempted to push our method to achieve the most refined results; nor have we included any global applications, although they are the motivation behind Theorem \ref{main}. Our proof of Theorem \ref{main} is different from that of the corresponding result in Bellovin-Gee, and when specialized to $G= \GL_n$, it is different from the arguments of \cite{gee11}, \cite{choi09}, and \cite{blggt14}, so we hope it may still be of independent interest. In the remainder of this introduction, we will briefly describe the approach to Theorem \ref{main}.

Common to all of the results related to Theorem \ref{main} is the translation of problems about Galois representations to problems about Weil-Deligne representations. In \cite{gee11}, \cite{bg17}, and our Theorem \ref{main}, the main ($\ell \neq p$) result is reduced using ideas of \cite{kisin08} to a corresponding ``unobstructedness" result in an appropriate moduli space of Weil-Deligne representations; the reduction in \cite{kisin08} is not obvious, but it carries over essentially formally to the settings of these subsequent papers. What is new in our paper is the proof of this unobstructedness result, Theorem \ref{thm:unobstructed}, which occupies \S \ref{sec:moduli}. In brief, there is a moduli space $X \to \Spec E$ (see \S \ref{sec:gwd}) of Weil-Deligne representations, and for any object $D_A \in X(A)$ an explicit complex $C^{\bullet}(D_A)$ whose cohomology controls the deformation theory of $D_A$. Taking $D_A= D_{\Phi, N, \tau}$ to be the universal object, our problem is to show that the cohomology measuring obstructions, $H^2(D_{\Phi, N, \tau})$, a coherent sheaf on $X$, vanishes away from a \textit{dense} open subscheme of $X$. 
The strategy initiated by Kisin shows this by studying an analogous moduli space $Y$ that parametrizes only a monodromy operator and an inertial type, together with the forgetful map $X \to Y$;  the essential and non-formal content of Theorem \ref{main} is the fact that every irreducible component of every non-empty fiber of this map contains an unobstructed point. In the papers \cite{kisin08}, \cite{gee11}, \cite{bellovin16}, and \cite{bg17}, this is achieved essentially by constructing one particular unobstructed point in each component. In contrast, our argument takes any point in such a fiber and connects it by a chain of $\Gm$'s and $\Ga$'s to an unobstructed point. It relies on a series of applications of the Jacobson-Morozov theorem.  This analysis is complicated by two factors.  Unlike the $\GL_n$ situation, the centralizer of the nilpotent monodromy operator may have multiple components, leading to additional components in the fibers.  Furthermore, the adjoint action of the image of Frobenius twists the inertial type (see condition \eqref{condlast} after Definition~\ref{defn:gwd}), while the analogous condition in the $\ell = p$ case is that the inertial type is preserved.  The latter condition defines a subgroup, while the former does not, complicating the argument.

\section{G-Weil-Deligne Representations and Deformation Theory} \label{sec:gwd}

Let $\ell$ and $p$ be distinct primes.  Let $K$ be an $\ell$-adic field, with residue field of order $q=\ell^f$, and let $W_K$ be the Weil group of $K$.  Let $\| \cdot \| : W_K \to W_K/I_K \simeq q^\ZZ $ denote the homomorphism which sends any geometric Frobenius to $q^{-1}$. Throughout this paper, we will fix a geometric Frobenius $\phi \in W_K$. Next consider a finite extension $E$ of $\QQ_p$ with ring of integers $\O_E$ and residue field $k$, and let $G$ over $\O_E$ be a smooth group scheme with reductive generic fiber; for the purposes of the next two sections, we only use the generic fiber, whereas in \S \ref{sec:def} we need a smooth group scheme over $\O_E$ in order to study the deformation theory of $\rhobar$.

\begin{defn} \label{defn:gwd}
For an $E$-algebra $A$, a \emph{$G$-Weil-Deligne representation} over $A$ is a trivial $G$-bundle $D_A$, a homomorphism $r : W_K \to \Aut_G(D_A)$ whose restriction to $I_K$ factors through a finite quotient, and an $N \in \Lie \Aut_G(D_A)$ such that $\Ad r(g) N = \|g\| N$.
\end{defn}

More concretely, if the inertial action factors through a fixed finite extension $L/K$, a $G$-Weil-Deligne representation over $A$ is given by a trivial $G$-bundle $D_A$ equipped with a homomorphism $\tau : I_{L/K} \to \Aut_G(D_A)$, an element $N \in \Lie \Aut_G(D_A)$, and an invertible $\Phi \in \Aut_G(D_A)$ that satisfy:
\begin{enumerate}
 \item $\Ad(\tau(g)) N = N$ for any $g \in I_{L/K}$. \label{condfirst}
 \item $\Ad(\Phi) N = q^{-1} N$. \label{condsecond}
 \item $\Ad(\Phi) \tau(\gamma) = \tau(\phi \gamma \phi^{-1})$ for all $\gamma \in I_{L/K}$. \label{condlast}
\end{enumerate}
We take $\Phi = r(\phi)$ and $\tau = r|_{I_K}$: recall that $\phi$ is the fixed geometric Frobenius in $W_K$.  

\begin{remark} \label{remark:nilpotent}
When $A$ is a field, note that condition \eqref{condsecond} implies that $N$ is nilpotent.  In particular, for each finite dimensional $G$-module, $N$ acts nilpotently on that module since \eqref{condsecond} forces the eigenvalues to be zero.  
\end{remark}

Let $\WD$ denote the category whose objects are pairs consisting of an $E$-algebra $A$ and a $G$-Weil-Deligne representation $(D_A, r, N)$ over $A$; a morphism $(A, D_A, r, N) \to (A', D_{A'}, r', N')$ is an $E$-algebra map $A \to A'$ and an isomorphism $D_A \otimes_A A' \xrightarrow{\sim} D_{A'}$ of $G$-bundles intertwining the actions of $(r,N)$ and $(r', N')$. We may regard $\WD$ as a category cofibered in groupoids over the category of $E$-algebras. We define $\WDL$ as the analogous category in which the $I_K$-action factors through the fixed finite quotient $I_{L/K}$.

\begin{remark}
We can identify $\Aut_G(D_A)$ with $G$ by picking a trivializing section, obtaining a homomorphism $r : W_K \to G(A)$.  This is analogous to the way a free $A$-module of rank $d$  with action of a group $\Gamma$ can be identified with a homomorphism $\Gamma \to \GL_d(A)$ after a choice of basis.
\end{remark}

We wish to study deformation theory in $\WDL$.  Let $\ad D_A := \Lie \Aut_G D_A$; the adjoint action gives an action of $I_{L/K}$ and of $\Phi$ on $\ad D_A$.  Consider the anti-commutative diagram
\[
\xymatrix{
(\ad D_A)^{I_{L/K}} \ar[rr]^{1-\Ad \Phi }\ar[d]^{\ad N} & & (\ad D_A)^{I_{L/K}} \ar[d]^{\ad N} \\
(\ad D_A)^{I_{L/K}} \ar[rr]^{q \Ad \Phi -1} & & (\ad D_A)^{I_{L/K}}.
}
\]
Let $C^\bullet(D_A)$ denote the total complex (indexed so as to be in degrees $0$, $1$, and $2$) associated to this double complex, and let $H^i(D_A)$ denote the $i$th cohomology. Note that the construction of $H^2(D_A)$ commutes with arbitrary base change.

Let $A$ be an Artin local $E$-algebra with maximal ideal $\m_A$, and let $I \subset A$ be an ideal with $I \m_A=0$.  Let $D_{A/I} \in \WDL(A/I)$, and define $\overline{D} = D_{A/I} \tensor{A/I} A/\m_A$.  Two liftings $D_A$ and $D'_A$ are equivalent if there exists a map $D_A \to D'_A$ of $G$-torsors compatible with the Weil-Deligne structure that reduces to the identity modulo $I$.

\begin{prop}\label{WDdef}
 If $H^2(\overline{D})=0$, then a lift $D_A \in \WDL(A)$ of $D_{A/I}$ exists.  The set of equivalence classes of liftings of $D_{A/I}$ to $A$ is a (possibly empty) torsor under $H^1(\overline{D}) \tensor{A/m_A} I$.
\end{prop}

\begin{proof}
The proof is essentially the same as that of \cite[Proposition 3.2]{bellovin16} and \cite[Lemma 2.1.1]{gee11}.
\end{proof}

We let $X$ denote the functor on the category of $E$-algebras which associates to an $E$-algebra $A$ the set of possible triples 
\[
(\Phi,N,\tau) \in G(A) \times \g_A \times \Rep_{A}(I_{L/K})
\]
that satisfy \eqref{condfirst}-\eqref{condlast} of \S\ref{sec:gwd}.  This is represented by a finite type locally closed subscheme of the space of all possible triples obtained by imposing the conditions.  Likewise, we consider a functor $Y$ consisting of pairs $(N,\tau)$ satisfying condition (1).  We also have $X_{N,\tau}$, $X_\tau$, and $Y_\tau$, where the subscripts denote a fixed choice of that variable. 
There are natural (forgetful) maps between these spaces.  Any object of $X$ defines a $G$-Weil-Deligne representation by viewing $G$ as a trivial $G$-torsor.

We write $D_{\Phi,N,\tau}$ for the universal triple on $X=\Spec R$, and define the sheaf $\Hcal$ on $X$ as the cokernel of
\begin{equation} \label{eq:hcaldef}
 q \Ad \Phi - 1 \oplus \ad N:  \g_R^{I_{L/K}} \oplus \g_R^{I_{L/K}} \to \g_R^{I_{L/K}}
\end{equation}
It is a coherent sheaf on $X$, and by semi-continuity the locus where it vanishes is open.  For a closed point $x$ with residue field $A$ corresponding to a $G$-Weil-Deligne representation $D_A$, we see that $\Hcal_x \simeq H^2(D_A)$.
We say a point $x$ is \emph{unobstructed} if $\Hcal_x = 0$.

The main technical result we will prove is the following:

\begin{thm}~\label{thm:unobstructed}
The unobstructed points are dense in $X$.
\end{thm}

The argument is inspired by the proofs of \cite[Lemma 3.1.5]{kisin08}, \cite[Lemma 2.1.3]{gee11}, and \cite[Proposition 5.2]{bellovin16}.  We will show that the unobstructed points are dense in the fiber $X_{N,\tau}$ of $X \to Y$ over any fixed $\tau$ and $N$.  As being unobstructed is an open condition, it suffices to find a single unobstructed point in each irreducible component of the fiber.  The proofs in these papers (and the proof of \cite[Theorem 2.3.6]{bg17}) proceed by writing down a point in each component, either by hand or by using the theory of associated cocharacters, and then directly verifying that the point is unobstructed.  In contrast, we start with an arbitrary point and connect it by a succession of $\Ga$'s and $\Gm$'s to an unobstructed point.  We will carry this out in the next section.

\section{Moduli Spaces for G-Weil-Deligne Representations}\label{sec:moduli}

Let $e$ be a nilpotent element in $\Lie G_{\overline{E}}$.  The proof of Theorem~\ref{thm:unobstructed} repeatedly uses cocharacters that interact well with $N$.

\begin{defn}
For a subgroup $H \subset G_{\overline{E}}$ with $e \in \Lie H$, a cocharacter $\lambda: \Gm \to H$ is adapted to $e$ provided that $\Ad \lambda(t) e = t^2 e$.
\end{defn}

As we are in characteristic zero, the nilpotent $e$ can be extended to an $\sl_2$-triple by the Jacobson-Morozov theorem.  If we exponentiate the triple and restrict to the diagonal $\Gm$, we obtain a cocharacter adapted to $e$.  This process also provide examples of the associated cocharacters used in \cite{bellovin16}.  Note that a general cocharacter adapted to $N$ is not necessarily associated to $N$.

Let $H$ be a reductive subgroup of $G_{\overline{E}}$, and set $\h = \Lie H$.  Consider a semi-simple $g \in G(\overline{E})$  such that $\Ad g (\h) = \h$, and a non-zero nilpotent $e \in \h$ such that $ \Ad(g) e = \alpha e$ for some non-zero $\alpha \in \overline{E}$.  We will find a cocharacter adapted to $N$ that interacts well with $g$.

\begin{lem} \label{lem:cocharlem}
There exists a cocharacter $\lambda: \Gm \to H^\circ$ adapted to $e$ such that for all $t$
\[
 \Ad \lambda(t) \Ad g = \Ad g \Ad \lambda(t).
\]
\end{lem}

\begin{proof}
We learned this argument from \cite[Lemma 2.1]{gross-reeder}. Consider the adjoint action of $g$ on the Lie algebra $\h$.  We will construct the desired cocharacter by constructing an $\sl_2$-triple compatible with the eigenspace decomposition 
\[
 \h = \bigoplus_{\mu} \h(\mu),
\]
where $\h(\mu)$ is the eigenspace for $\Ad g$ with eigenvalue $\mu$.  (As $g$ is semi-simple, we do not need generalized eigenspaces.)  The Jacobson-Morozov theorem gives an $\sl_2$-triple $(e, h_0,  f_0)$ in $\h$.  Note that $e \in \h(\alpha)$.  For any $h_\mu \in \g(\mu)$, $$\Ad(g)[h_\mu , e] = [\Ad(g) h_\mu , \Ad(g) e] = \mu \alpha [h_\mu,e],$$
so $[h_\mu,e]$ is an element of $\h( \mu \alpha)$.  Now decompose
$h_0 = h + h'$ with $h \in \h(1)$ and $h' \in \bigoplus_{\mu \neq 1} \h(\mu)$.  Using that $2 e = [h_0,e] = [h,e] + [h',e] \in \h(\alpha)$ and considering eigenspaces, we conclude that $2e = [h,e]$.  In particular, $\h(1) \neq 0$.  Likewise, we decompose $f_0 = f + f'$ with $f \in \h(\alpha^{-1})$ and $f' \in \bigoplus_{\mu \neq \alpha^{-1}} \h(\mu)$ and keep track of eigenspaces in the relation $h_0 = [e,f_0] = [e,f] + [e,f']$ to conclude that $h = [e,f]$.  Likewise we see that $[h,f] = - 2 f$, so $(e, h, f)$ is another $\sl_2$-triple; to obtain the cocharacter adapted to $e$, we exponentiate and restrict to the diagonal $\Gm$. Then we see that $\Ad g \Ad \lambda(t) = \Ad \lambda(t) \Ad g$ as $h \in \h(1)$.
\end{proof}

We now proceed with the proof of Theorem~\ref{thm:unobstructed}. 

\begin{proof}
Fix $\tau \colon I_{L/K} \to G(E)$. It suffices to show every non-empty fiber $X_{N,\tau}$ of the forgetful map $X_\tau \to Y_\tau$ contains a dense open subset on which $\Hcal$ vanishes. As this is an open condition, it suffices to find a single closed point $x$ in each irreducible component of $X_{N,\tau}$ for which $\Hcal_x=0$.

Consider the fiber $X_{N,\tau}$ over a fixed $N \in (\Lie G)(E')$ for some finite extension $E'/E$; assume it is non-empty, so that we have additionally a $\Phi \in G(\overline{E})$ such that the relations (1)-(3) of \S \ref{sec:gwd} hold for $\Phi, N, \tau$; in this particular fiber we are only allowed to vary $\Phi$. To find an unobstructed point in the fiber, we may work over $\overline{E}$, and to simplify notation, we change notation and let $G$ (and similarly $N$) be defined over $\overline{E}$.  Fix a square root of $q$ in $\overline{E}$.  By Remark~\ref{remark:nilpotent}, $N$ is nilpotent.

We break into cases depending on whether $N = 0$.  In both cases, the strategy is to connect $\Phi$ by $\Ga$'s and $\Gm$'s to an unobstructed point.  A basic but important observation is that the fiber $X_{N,\tau}$ is a  $Z_G(N)\cap Z_G(\tau)$-torsor.  As the fiber is a torsor under $Z_G(N) \cap Z_G(\tau)$, the irreducible components are the same as connected components.  Furthermore, for homomorphisms $\lambda: \Gm \to Z_G(N) \cap Z_G(\tau)$ or $\psi: \Ga \to Z_G(N) \cap Z_G(\tau)$, $\Phi \lambda(t)$ and $\Phi \psi(s)$ are in the same component of the fiber as $\Phi$.  

\textbf{Case 1}:  When $N=0$, to be unobstructed means that $q \Ad \Phi - 1$ is invertible on $\g^{I_{L/K}}$.  Given $\Phi \in X_{0,\tau}(\overline{E})$ such that $q \Ad \Phi - 1$ is not invertible, we will find a $\Phi'$ in the same component as $\Phi$ such that the generalized eigenspace of $\Ad \Phi'$ with eigenvalue $q^{-1}$ has smaller dimension. By induction, this produces an unobstructed point.  

By hypothesis, there is an $N' \in \g^{I_{L/K}}$ such that $q \Ad(\Phi) N' =  N'$.  Note that $N'$ is nilpotent as it is conjugate to $q^{-1} N'$ and so acts nilpotently on every finite-dimensional $G$-module by consideration of eigenvalues.   

\begin{lem} 
There exists $\Phi_s$ in the same component of the fiber $X_{0, \tau}$ as $\Phi$ such that $\Phi_s$ is semi-simple.  There is a cocharacter $\lambda: \Gm \to Z_G(\tau)$ adapted to $N'$ such that $$\Ad \Phi_s \Ad \lambda(t) = \Ad \lambda(t) \Ad \Phi_s.$$
\end{lem}

\begin{proof}
Consider the Jordan decomposition $\Phi = \Phi_s \Phi_u$ where $\Phi_s \in G(\overline{E})$ is semi-simple and $\Phi_u \in G(\overline{E})$ is unipotent.  The relation $\Ad \Phi (\tau(\gamma)) = \tau(\phi \gamma \phi^{-1})$ for $\gamma \in I_{L/K}$ implies that there is an integer $n$ such that $\Phi^n \in Z_G(\tau)(\overline{E})$.  As $\Phi_s$ and $\Phi_u$ commute, we see that $\Phi^n = \Phi_s^n \Phi_u^n$.  This is also the Jordan decomposition for $\Phi^n$ in $Z_G(\tau)$.  Since Jordan decomposition is compatible with inclusions of groups, we see $\Phi_u^n \in Z_G(\tau)(\overline{E})$.  As $\Phi_u^n$ is unipotent, we may write it as $\Phi_u^n = \exp ( n Y)$ for a nilpotent $Y \in \Lie Z_G(\tau)$ (for any unipotent group $U$ in characteristic zero, there is an isomorphism of schemes $\exp \colon \mathrm{Lie}(U) \to U$ induced by embedding $\mathrm{Lie}(U)$ into some Lie algebra of strictly upper-triangular nilpotent matrices, and then applying the usual power series of the exponential; if $U$ is commutative, $\exp$ is moreover an isomorphism of group schemes).  This shows that $\Phi_u = \exp(Y)$ lies in $Z_G(\tau)$.  As the fiber is a $Z_G(\tau)$-torsor, we see that $\Phi_s \exp(t Y)$ lies in the fiber for every $t$, and hence that $\Phi$ and $\Phi_s$ are in the same component.

The second statement is Lemma~\ref{lem:cocharlem} applied to $Z_G(\tau) \subset G$; note that $Z_G(\tau)$ is reductive by repeatedly applying the fact that for a reductive group $H$ and a semisimple element $h \in H$, $Z_H(h)$ is reductive \cite[Theorem 2.2]{humphreys95}.
\end{proof}

We will conclude the proof of Case 1 by showing that for a generic choice of $t$, the dimension of the $q^{-1}$-eigenspace of $\Ad (\Phi_s \lambda(t))$ on $\g^{I_{L/K}}$ is less than the dimension of the $q^{-1}$-eigenspace of $\Ad \Phi_s$.  
This will suffice, as $\Phi$, $\Phi_s$, and $\Phi_s \lambda(t)$ all lie in the same component of $X_{0, \tau}$, and the $q^{-1}$-generalized eigenspace for $\Ad \Phi$ has the same dimension as the $q^{-1}$-eigenspace for $\Ad \Phi_s$.  

We decompose 
\[
 \g^{I_{L/K}} = \bigoplus_{\mu} \g^{I_{L/K}}(\mu)
\]
where $\g^{I_{L/K}}(\mu)$ is the $\mu$-eigenspace of $\Ad \Phi_s$.  Now $\Ad \lambda(t)$ preserves $\g^{I_{L/K}}(\mu)$ as $\Ad \Phi_s$ and $\Ad \lambda(t)$ commute.  On the finitely-many non-zero $\g^{I_{L/K}}(\mu)$ with $\mu \neq q^{-1}$, the condition that $\Ad \Phi_s \lambda(t)$ not have $q^{-1}$ as an eigenvalue is simply the condition that $\Ad \lambda(t)$ not have $q^{-1} \mu^{-1}$ as an eigenvalue.  This is a non-empty, open condition (consider $t=1$).  Furthermore, we compute that
\[
 \Ad (\Phi_s \lambda(t)) N' = q^{-1} t^2 N'.
\]
So if $t \neq \pm 1$, we see that the eigenvalue for $N'$ is not $q^{-1}$.  Thus for a generic choice of $t$, $\Ad (\Phi_s \lambda(t))$ has a smaller $q^{-1}$-eigenspace than $\Ad (\Phi_s)$, and we conclude by induction.

\textbf{Case 2}:  The case $N \neq 0$ follows the same strategy, but is more involved.  
We will first find a semi-simple point in a given component of the fiber, and then modify it using cocharacters valued in $Z_G(N) \cap Z_G(\tau)$ so it is unobstructed.  A key technique is passing between points in the fiber and points of $Z_G(N)$: for any cocharacter $\lambda$ adapted to $N$ and $\Phi$ in the fiber, we see $\Phi \lambda(q^{1/2}) \in Z_G(N)$ as 
\[
 \Ad (\Phi \lambda(q^{1/2})) N = q^{-1} (q^{1/2})^2 N = N
\]

\begin{lem}\label{generalss}
In each non-empty component of the fiber $X_{N,\tau}$ of $X_\tau \to Y_\tau$ above $N$, there exists a semi-simple point $\Phi$. There is a cocharacter $\lambda : \Gm \to Z_G(\tau)^\circ$ adapted to $N$ such that for all $t$
\[
 \Ad \Phi \Ad \lambda(t) = \Ad \lambda(t) \Ad \Phi.
\]
\end{lem}

\begin{proof}
Let $\Phi'$ be a point in the desired component of the fiber.  Let $\lambda' : \Gm \to Z_G(\tau)$ be any cocharacter adapted to $N$.  It is easy to check that $\Psi := \Phi' \lambda'(q^{1/2}) \in Z_G(N)(\overline{E})$ satisfies $\Ad(\Psi) \tau(\gamma) = \tau(\phi \gamma \phi^{-1})$
for $\gamma \in I_{L/K}$.  As before, some power $\Psi^n$ is in $Z_G(N) \cap Z_G(\tau)(\overline{E})$. 

Consider the Jordan decomposition $\Psi = \Psi_u \Psi_s$ in $Z_G(N)$ with $\Psi_u$ and $\Psi_s$ commuting unipotent and semi-simple elements.  As $\Psi^n = \Psi_u^n \Psi_s^n$ is a Jordan decomposition for $\Psi^n$ and Jordan decomposition is compatible with inclusions of groups, we see that $\Psi_u^n \in Z_G(N) \cap Z_G(\tau) (\overline{E})$.  As $\Psi_u$ is unipotent, it follows that $\Psi_u \in Z_G(N) \cap Z_G(\tau) (\overline{E})$.  This shows that
\[
 \Ad(\Psi_s \lambda'(q^{-1/2})) N = q^{-1} N \quad \quad \text{and} \quad \quad \Ad( \Psi_s \lambda'(q^{-1/2})) \tau(\gamma) = \tau(\phi \gamma \phi^{-1})
\]
for $\gamma \in I_{L/K}$.  Note that $\Phi'$ and$\Psi_s \lambda'(q^{-1/2})= \Phi' \lambda'(q^{1/2})\Psi_u^{-1}\lambda'(q^{-1/2})$ lie in the same component of the fiber.

The above used an arbitrary cocharacter adapted to $N$.  Using Lemma~\ref{lem:cocharlem}, now pick a cocharacter $\lambda \colon \Gm \to Z_G(\tau)$ adapted to $N$ whose adjoint action commutes with that of $\Psi_s$.  Define $\Phi = \Psi_s \lambda(q^{-1/2})$.  As before, we check that $\Phi$ is in $X_{N, \tau}(\overline{E})$.  Furthermore, note that $\Phi$ and $\Psi_s \lambda'(q^{-1/2})$ lie in the same component of the fiber, as the family
\[
 \Phi_t:= \Psi_s \lambda(q^{-1/2} t) \lambda'( t^{-1})
\]
interpolates between them.  Finally, $\Phi$ is semi-simple as the adjoint actions of $\Psi_s$ and $\lambda(q^{-1/2})$ commute and are semi-simple, and it is clear that the adjoint actions of $\Phi$ and $\lambda(t)$ commute.
\end{proof}

Recall that a point $(\Phi',N)$ is unobstructed if 
\[
q \Ad \Phi' - 1 \oplus \ad N : \g^{I_{L/K}} \oplus \g^{I_{L/K}} \to \g^{I_{L/K}}
\]
is surjective.  If $\Phi'$ is semi-simple, then so is $q \Ad \Phi' - 1$.  Hence $\g^{I_{L/K}}$ is a direct sum of the kernel and the image of $q \Ad \Phi' - 1$.  In this case, $(\Phi',N)$ is unobstructed provided that $\ker( q \Ad \Phi' - 1 ) \subset \Im(\ad N)$. We now continue with the $\Phi$ and $\lambda$ produced by Lemma \ref{generalss}.  

We decompose
\[
 \g^{I_{L/K}} = \sum_{n \in \ZZ} \g_n
\]
where $\g_n$ is the space where $\Ad \lambda(t)$ acts by $t^n$.  For $n \leq 0$, we set $\g_n^{\lw} := \ker (\ad (N) |_{\g_n}^{-n+1})$ and $\g'_n = \Im (\ad(N) |_{\g_{n-2}})$, and decompose
\[
 \g_n = \g_n' \oplus \g_n^{\lw}.
\]
Here $\g_n^{\lw}$ are the lowest weight vectors of the $\sl_2$-triple containing $N$ that was used to define $\lambda$.  Note $\ad N$ gives an isomorphism between $\g_{n-2}$ and $\g'_n$.

\begin{lem}
For any $n \leq 0$, we have $\Ad \Phi (\g_n^{\lw}) = \g_n^{\lw}$ and $\Ad \Phi( \g_n') = \g_n'$.
\end{lem}

\begin{proof}
First note that $\Ad(\Phi)$ preserves $\g_n$ since $\Ad(\Phi)$ and $\Ad(\lambda(t))$ commute. For $v \in \g_n^{\lw}$, we compute that
\[
 (\ad N)^{-n+1} \left( \Ad (\Phi) v \right) = q^{-n+1} \Ad \Phi\left( ( \ad N)^{-n+1} v \right) = 0.
\]
Since $\Phi$ acts invertibly ($\Ad \Phi^{-1}$ is an inverse), this gives the first equality.
For the second, consider $v \in \g_n'$.  Writing $v = \ad N (v')$ for $v' \in \g_{n-2}$, we compute that
\[
 \Ad \Phi(v) = [\Ad \Phi( N), \Ad \Phi( v')] = q^{-1} \ad N ( \Ad \Phi(v')).
\]
Thus $\Ad \Phi (v) \in \ad N (\g_{n-2}) = \g'_n$.  Since $\Phi$ acts invertibly, we are done.
\end{proof}

Note that all $\g_n$ with $n>0$ lie in the image of $\ad N$, as do $\g'_n$ for $n \leq 0$.  To check that $\ker( q \Ad \Phi' - 1 ) \subset \Im(\ad N)$, by the Lemma it suffices to show that $\Ad \Phi'$ does not have eigenvalue $q^{-1}$ on any of the $\g_n^{\lw}$.  We will modify $\Phi$ so this holds.

Consider the element $\Psi_s= \Phi \lambda(q^{1/2})$. There are two cases to consider, depending on whether $\Ad(\Psi_s)$ has infinite or finite order. First suppose it has finite order, say $m$, so that $\Ad(\Phi)^m= \Ad(\lambda(q^{1/2}))^{-m}$. Consider any eigenvector $v \in \mathfrak{g}^{\lw}_n$ (for some $n \leq 0$) of $\Ad(\Phi)$, with eigenvalue $\alpha$. Then
\[
\alpha^m v =\Ad(\Phi)^m (v)= \Ad(\lambda(q^{-m/2}))v= q^{-mn/2} v,
\]
and we clearly cannot have $\alpha=q^{-1}$ for $n \leq 0$.

Next consider the case where $\Ad(\Psi_s)$ has infinite order. The group
\[
\mathcal{Z}= \{g \in Z_G(N) \cap Z_G(\tau): \text{$\Ad(g)$ and $\Ad(\Phi)$ commute}\}
\]
contains a non-trivial power of the semi-simple, infinite order, element $\Psi_s$ (since $\Ad \Phi^m \tau(\gamma) = \tau(\gamma)$ when conjugation by $\phi^m$ is trivial on $I_K$).
Some power $\Psi_s^m$, with $m>0$, is then contained in a non-trivial torus of $\mathcal{Z}$, and we let $\lambda''$ be any co-character of (this torus of) $\mathcal{Z}$ whose image contains $\Psi_s^m$: certainly in any torus every element is in the image of some co-character.
We consider the adjoint action of $\Phi_t := \Phi \lambda''(t)$ and claim that for a generic choice of $t$, the point $\Phi_t$ will be an unobstructed point of the fiber of $X_\tau \to Y_\tau$ over $N$.

As the fiber is a $Z_G(N) \cap Z_G(\tau)$-torsor, $\Phi_t$ lies in the desired fiber.  Since $\Ad \Phi$ and $\Ad \lambda''(t)$ commute, they have common eigenvectors.  Let $v$ be an eigenvector of $\Ad \Phi$ in one of the $\g_n^{\lw}$. 
\begin{itemize}
 \item In the case that $\Ad \lambda''(t) v \neq v$ for some $t$, it is a non-empty open condition for $\Ad \Phi_t (v) \neq q^{-1} v$. 
 
 \item  If $\Ad \lambda''(t) v = v$ for all $t$, then as $\Psi_s$ is in the image of $\lambda''$ we see that
\[
v =  \Ad \Psi_s (v) = \Ad \Phi \lambda(q^{1/2}) (v) = q^{n /2} \Ad \Phi (v) .
\]
In particular, $\Ad \Phi (v) \neq q^{-1} v$ as $n /2 \leq 0$, and hence $\Ad \Phi_t (v) \neq q^{-1}v$ for all $t$.  
\end{itemize}
Combining these conditions for each eigenvector in some $\g_n^{\lw}$, for a generic choice of $t$ we see $\Ad \Phi_t$ does not have eigenvalue $q^{-1}$ on any of the $\g_n^{\lw}$.  In that case $\Phi_t$ is an unobstructed point of the $X_{N,\tau}$, completing the case that $N \neq 0$.
\end{proof}

\section{Analysis of Local Galois Deformation Rings}\label{sec:def}

As before, let $\ell$ and $p$ be distinct primes.  Let $K$ be an $\ell$-adic field. In this section we will apply the local monodromy theorem, so we fix a compatible collection of $p$-power roots of unity in $\overline{K}$, yielding as usual a surjection $t_p : I_K^t \onto \ZZ_p(1)$ from the tame inertia group of $K$.  Let $E$ be a finite extension of $\QQ_p$ with ring of integers $\O_E$ and residue field $k$ of size $q = \ell^f$. Let $\mathcal{C}_{\O_E}$ denote the category of complete local noetherian $\O_E$-algebras with residue field $k$. Let $G$ be a smooth group scheme over $\O_E$ such that $G_E$ is reductive.  Fix a continuous homomorphism $\rhobar : \Gal(\overline{K}/K) \to G(k)$. Consider the morphism $D^\square_{\rhobar} \to D_{\rhobar}$ of (categories cofibered in) groupoids over $\Ccal_{\O_E}$,
where $D^\square_{\rhobar}(R)$ is the category (set) of lifts of $\rhobar$ to $G(R)$, and $D_{\rhobar}(R)$ is the category whose objects are lifts of $\rhobar$, and where a morphism between lifts $\rho$ and $\rho'$ is an element $g \in \widehat{G}(R)$ such that $g \rho g^{-1}= \rho'$. We are interested in the generic fiber $R_{\rhobar}^{\square}[\frac{1}{p}]$ of the universal lifting ring $R_{\rhobar}^{\square}$ (representing $D^\square_{\rhobar}$); when the corresponding deformation functor for $\rhobar$ is also representable, we obtain analogous results for $R_{\rhobar}[\frac{1}{p}]$.  We will analyze $R_{\rhobar}^{\square}[\frac{1}{p}]$ by means of $G$-Weil-Deligne representations.

Let $A^\circ$ be a complete local noetherian $\mathcal{O}_E$-algebra that has no $p$-torsion, with generic fiber $A= A^\circ[\frac{1}{p}]$, and with a continuous homomorphism $\rho \colon \Gamma_K \to G(A^\circ)$.  We can associate a $G$-Weil-Deligne representation $D_A$ to $\rho$ using the following construction (compare \cite[Proposition 4.1.6]{emerton-helm}).  

For each finite dimensional $E$-linear representation $M$ of $G$, we obtain a representation of $\Gamma_K$ on $M_{A^\circ}$.  Let $\m$ be the maximal ideal of $A^\circ$.  Let $e$ be the $\m$-adic valuation of $p$, and fix an integer $j$ such that $j > \frac{e}{p-1}$.
Since $\rho$ is continuous and $M_{A^\circ} / \m^j M_{A^\circ}$ is discrete in the $\m$-adic topology, there is a compact open normal subgroup $H_{M} \subset I_K$ that acts trivially on $M_{A^\circ} / \m^j M_{A^\circ}$.  Doing so for a faithful representation, we may find a common $H$ that works for every choice of $M$ and pick an element $\alpha \in H$. Note the kernel of the reduction map $\Aut(M_{A^\circ / \m^i}) \to \Aut(M_{A^\circ/\m^j}) $ is a $p$-group for $i>j$, so the action of $H$ factors through the chosen $t_p : I_K^t \to \ZZ_p(1)$.

Now consider the natural representation $\rho_M : \Gamma_K \to \Aut(M_{A})$ over $A$. 
The denominators in the power series for $\log$ exist in $A$ (since $p$ is inverted), and the power series for $\log( \rho_M(g))$ for $g \in H$
converges in the $\m$-adic topology on $\End(M_{A})$ since $\rho_M(g) \equiv \Id \pmod{\m^j}$.  Furthermore, we see that
$\exp( \log(\rho_M(g)))$ exists and equals $\rho_M(g)$ for $g \in H$ since the power series for the exponential converges.  (The $\m$-adic valuation of $\log(\rho_M(g))$ is greater than $\frac{e}{p-1}$ and the valuation of $n!$ is at most $\frac{e \cdot n}{p-1}$.)

We set $N_M = \frac{1}{t_p(\alpha)}\log( \rho_M(\alpha))$ and define $r_M: W_K \to \Aut(M_{A})$ by
\[
r_M ( \phi^n \sigma) = \rho_M(\phi^n \sigma) \exp(-t_p(\sigma) N_M)
\]
where $\phi$ is Frobenius and $\sigma \in I_K$.  Note that $\exp(-t_p(\sigma) N)$ exists since it can be rewritten as $\exp( \log( \rho_M(\alpha^{-t_p(\sigma)})))$ and $\alpha^{-t_p(\sigma)}\in H$.  We see that $r_M$ is trivial on $H$.  Furthermore, for $g= \phi^n \sigma \in W_K$ we compute that
\[
\Ad( r_M(g)) N_M= \frac{1}{t_p(\alpha)} \log(\rho_M( g \alpha g^{-1})) = \frac{1}{t_p(\alpha)} \log(\rho_M(\phi^n \sigma \alpha \sigma^{-1} \phi^{-n} ))
\]
since $\exp(N_M)$ commutes with $N_M$.  Now $\rho_M(\sigma \alpha \sigma^{-1}) = \rho_M(\alpha)$ as $\rho_M$ factors through the abelian $\ZZ_p(1)$.  Furthermore, conjugation by $\phi$ is multiplication by $q^{-1}$ on $\ZZ_p(1)$.  This shows that
$$\Ad(r_M(g)) N_M= \|g \| N_M.$$ 
A similar calculation shows that $r_M$ is a homomorphism.

The data of $r_M$ and $N_M$ for every finite dimensional $E$-linear representation $M$ of $G$ gives us a homomorphism $r : W_K \to G(A)$ that is trivial on $H$ and an $N \in (\Lie G)(A)$ by a Tannakian argument as in \cite[Appendix A]{bellovin16}.  In particular, \cite[A.2.4]{bellovin16} discusses how to deal with $N$, while Section A.2.6 discusses how to deal with the representation of a group.
Furthermore, we have that
\[
\Ad r(g) N = \|g\| N,
\]
as we have checked it on each representation $M$.  In other words, we have a $G$-Weil-Deligne representation over $A$ (with the $G$-bundle canonically trivialized).

\begin{remark}\label{GWD}
Suppose $G = \GL_n$, $E'$ is a finite extension of the field $E$, and $x$ is an $E'$-valued point of $A$.  Specializing $D_A$ at $x$ gives the Weil-Deligne representation associated to $\rho_x$ using the standard construction.
\end{remark}

For the remainder of this section, fix a finite extension $L/K$ and an homomorphism $\tau \colon I_{L/K} \to G(E)$ that arises as the restriction to $I_K$ of a $G$-Weil-Deligne representation. We call a $G(\overline{E})$-conjugacy class of such $\tau$ an \textit{inertial type}. For any artin local $E$-algebra $B$, with residue field some finite extension $E'$ of $E$, we say a continuous homomorphism $\rho \colon \Gamma_K \to G(B)$ is type $\tau$ if the associated $G$-Weil-Deligne representation $(r, N)$ has inertial restriction with reduction $$I_{L/K} \xrightarrow{r} G(B) \to G(E')$$ being $G(\overline{E})$-conjugate to $\tau$.  This condition is equivalent to $\tau \colon I_{L/K} \to G(B)$ being $G(B \otimes_{E'} \overline{E})$-conjugate to $\tau$:  this follows from standard deformation theory and the fact that $H^i(I_{L/K}, \ad(\rho))=0$ for $i=1, 2$. (We extend the definition of type to the case where $B$ is any finite $E$-algebra by imposing the above condition on each local factor of this artin ring.)

For any complete local noetherian $\mathcal{O}_E$-algebra $A^\circ$, with generic fiber $A= A^\circ[\frac{1}{p}]$, and equipped with a continuous homomorphism $\rho \colon \Gamma_K \to G(A^\circ)$, there is a quotient $A \onto A^\tau$, equal to a union of irreducible components of $\Spec A$, such that for any finite $E$-algebra $B$, an $E$-algebra map $A \xrightarrow{f} B$ factors through $A^\tau$ if and only if $f\circ \rho$ has inertial type $\tau$ (see the proof of \cite[Proposition 3.0.12]{balaji}). In particular, we can form the quotient $(R_{\rhobar}^\square[\frac{1}{p}])^\tau$. We remark here that with a slight addition to the argument of \cite[Proposition 3.0.12]{balaji}, we could equally well carry out the preceding discussion (and subsequent analysis) with inertial types defined to be $G^0(\overline{E})$-conjugacy classes rather than $G(\overline{E})$-conjugacy classes; we omit the details, but note that this would yield a slightly more refined result.

Our main result is the following:

\begin{thm}
$\Spec (R_{\rhobar}^{\square}[\frac{1}{p}])^\tau$ admits a regular, dense open subscheme, and it is equidimensional of dimension $\dim(G)$.
\end{thm}
\begin{proof}
Let $A= (R_{\rhobar}^{\square}[\frac{1}{p}])^\tau$, and let $A^\circ$ be the scheme-theoretic closure of $A$ in $R_{\rhobar}^{\square}$. The universal lift of $\rhobar$ induces a continuous homomorphism $\Gamma_K \to G(A^\circ)$, and so (note that $A^\circ$ is $p$-torsion free) we obtain a $G$-Weil-Deligne representation over $A$ by the construction preceding Remark \ref{GWD}. Denote it by $D_A \in WD_{G, L}(A)$. 
For any closed point $x$ of $\Spec A$, write $\m_x$ for the maximal ideal at $x$, $A_x$ for the completion of $A$ at $\m_x$, $E_x$ for the residue field, $\rho_x$ for the associated homomorphism $\Gamma_K \to G(E_x)$, and $D_x$ for the associated object of $WD_{G, L}(E_x)$. The argument of \cite[Proposition 3.3.1]{kisin08} implies that the associated morphism (of groupoids over $E$-algebras) 
\[
\Spec A_x \to WD_G
\]
is formally smooth; the hypothesis of \cite[Proposition 3.3.1]{kisin08} is just the assertion that $D^\square_{\rhobar} \to D_{\rhobar}$ is formally smooth, which is clear since $G$ is smooth.  Then the argument of \cite[Proposition 3.1.6]{kisin08} shows that if $U_x$ denotes the complement of the support of $H^2(D_{A_x})$ in $\Spec A_x$, then $U_x$ is formally smooth and open dense in $\Spec A_x$: here is where we crucially invoke our main result, Theorem \ref{thm:unobstructed}. 

Now let $U$ denote the complement of the support of $H^2(D_A)$.  As $U_x$ is the base change of $U$ to $A_x$, the density of $U_x$ in $A_x$ for all $x$ implies that $U$ is dense in $A$.  Furthermore, for any closed point $x \in U$, we know that  $U_x = A_x$ is formally smooth, hence regular.  Since $A$ is noetherian, this implies that the localization of $A$ at any such $x$ is regular, and hence that $U$ is regular.  

Now let $x$ be any closed point of $U \subset \Spec A$.  We know the completion $A_x$ is a regular local ring, and we will be done once we compute the $E_x$-dimension of its tangent space $\Hom_{E_x}(\m_x/\m_x^2, E_x)$. By a standard deformation theory argument (see \cite[Proposition 2.3.5]{kisin09}),
 $A_x$ pro-represents the functor of lifts of $\rho_x$ to artin local $E_x$-algebras with residue field $E_x$, and this tangent space is then isomorphic to the space of continuous 1-cocycles $Z^1(\Gamma_K, \ad(\rho_x))$, and therefore has dimension
\[
\dim A_x= \dim_{E_x} H^1(\Gamma_K, \ad(\rho_x))+ \dim(G)- \dim_{E_x}(\ad(\rho_x)^{\Gamma_K}).
\]
Now note that from the construction of $G$-Weil-Deligne representations from Galois representations, we obviously have that $H^0(D_x)= (\ad(\rho_x))^{\Gamma_K}$. Moreover, $H^1(\Gamma_K, \ad(\rho_x))$ classifies $E_x[\epsilon]$-deformations of $\rho_x$, and (by Proposition \ref{WDdef}), $H^1(D_x)$ classifies equivalence classes of lifts of $D_x$ to $WD_{G, L}(E_x[\epsilon])$. Since $D_x$ arises from the Galois representation $\rho_x$, it is a \textit{bounded} $G$-Weil-Deligne representation (i.e., in any finite-dimensional representation of $G$, the eigenvalues of $\Phi$ are $p$-adic units). This implies that for any lift of $D_x$ to $WD_{G, L}(E_x[\epsilon])$, the usual formula associating a Galois representation to a $G$-Weil-Deligne representation applies to yield a continuous lift of $\rho_x$ to $G(E_x[\epsilon])$. Conversely, such a lift can by the usual argument of the monodromy theorem be converted into a lift of $D_x$. These two procedures are inverses, identifying equivalence classes of $E_x[\epsilon]$-deformations, and they therefore identify the $E_x$-vector spaces $H^1(D_x)$ and $H^1(\Gamma_K, \ad(\rho_x))$. Now combining the local Euler-characteristic formula and the fact that the Euler characteristic of the complex $C^{\bullet}(D_x)$ obviously vanishes, we see that
\[
\dim A_x= \dim(G)+ \dim_{E_x}(H^2(D_x))= \dim(G),
\]
since $H^2(D_x)=0$ for $x \in U$.
\end{proof}

\begin{remark}
We note as a consequence of the proof that for all $x \in U$, the Galois cohomology group $H^2(\Gamma_K, \ad(\rho_x))$ in fact vanishes.
\end{remark}

\begin{remark} \label{rmk:formalsmoothness}
The results of \cite[Theorem 3.3.4]{kisin08}, and later papers like \cite[Theorem 2.1.6]{gee11} that adapt its arguments, state that $U$ is formally smooth over $\QQ_p$, not just regular.  This distinction is irrelevant for the applications to computing the dimension of the generic fiber, as all that is used is regularity, but regularity does not imply formal smoothness in this setting. For example, consider $A = \ZZ_p[[t]][ \frac{1}{p}] $.  As $\ZZ_p[[t]]$ is regular, so is its localization $A$. 
However, $A$ is not formally smooth over $\QQ_p$ in the discrete topology.
We thank Bhargav Bhatt for suggesting the following argument.  

Let $B= \ZZ_p[t][\frac{1}{p}]$.  The natural map $B \to A$ is regular: it is flat because $\ZZ_p[[t]]$ is flat over $\ZZ_p[t]$, and the fibers are points and hence geometrically regular.  Then the Jacobi-Zariski exact sequence for the maps $\QQ_p \to B \to A$ gives
\[
 0\to A \tensor{B} \Omega_{B/\QQ_p} \to \Omega_{A/\QQ_p} \to \Omega_{A/B}  \to 0. 
\]
Now $A \tensor{B} \Omega_{B/\QQ_p}$ is a free $A$-module with basis $dt$.
The above sequence is split by the element of $\Hom_A(\Omega_{A/\QQ_p},A \tensor{B} \Omega_{B/\QQ_p}) = \on{Der}_{\QQ_p}(A,A \tensor{B} \Omega_{B/\QQ_p})$ given by sending $f(t) \in A $ to $f'(t)dt$.  If $A$ were formally smooth over $\QQ_p$ in the discrete topology, $\Omega_{A/\QQ_p}$ would be projective over $A$.  Together with the splitting, this would imply that $\Omega_{A/B}$ embeds into a free $A$-module.  Now $\Omega_{A/\QQ_p}$  is also non-zero, as a localization is $\Omega_{\QQ_p((t))/ \QQ_p(t)}$ whose dimension over $\QQ_p(t)$ is the transcendence degree of the non-algebraic extension $\QQ_p((t))$ of $\QQ_p(t)$. 
But $\Omega_{A/B}$ is also $t$-divisible since $B/t \to A/t$ is an isomorphism, so it cannot embed into a free $A$-module.  This contradiction shows that $A$ cannot be formally smooth over $\QQ_p$ in the discrete topology.
\end{remark}

\providecommand{\bysame}{\leavevmode\hbox to3em{\hrulefill}\thinspace}
\providecommand{\MR}{\relax\ifhmode\unskip\space\fi MR }
\providecommand{\MRhref}[2]{%
  \href{http://www.ams.org/mathscinet-getitem?mr=#1}{#2}
}
\providecommand{\href}[2]{#2}

\end{document}